\documentclass[reqno, 12pt]{amsart}
\usepackage{amsmath,amssymb,amsfonts,latexsym,amscd,psfrag,mathabx,graphicx,mathrsfs,stmaryrd}
\usepackage[square,sort,comma,numbers]{natbib}

\usepackage[utf8]{inputenc}
\usepackage[T1]{fontenc}
\usepackage{lmodern}
\usepackage{setspace}
\usepackage{tikz}
\usepackage{verbatim}
\usetikzlibrary{shapes,shadows,calc}
\usepgflibrary{arrows}
\usetikzlibrary{arrows, decorations.markings, calc, fadings, decorations.pathreplacing, patterns, decorations.pathmorphing, positioning}
\tikzset{nodc/.style={circle,draw=blue!50,fill=pink!80,inner sep=4.2pt}}
\tikzset{nod1/.style={circle,draw=black,fill=black,inner sep=1pt}}
\tikzset{nod2/.style={circle,draw=black,fill=black,inner sep=2pt}}
\tikzset{nod3/.style={circle,draw=black,fill=black,inner sep=3pt}}
\tikzset{nodempty/.style={circle,draw=black,inner sep=2pt}}

\tikzset{nodde/.style={circle,draw=red!50,fill=red!80,inner sep=4.2pt}}
\tikzset{nodde1/.style={circle,draw=black!50,fill=black!80,inner sep=4.2pt}}
\tikzset{nodde2/.style={circle,draw=black!50,fill=purple!80,inner sep=4.2pt}}
\tikzset{nodde3/.style={circle,draw=black!50,fill=green!80,inner sep=4.2pt}}
\tikzset{nodde4/.style={circle,draw=black!50,fill=brown!80,inner sep=4.2pt}}

\tikzset{nodinvisible/.style={circle,draw=white,inner sep=2pt}}
\tikzset{nodpale/.style={circle,draw=gray,fill=gray,inner sep=1.6pt}}

\usepackage{scalefnt}
\usetikzlibrary{arrows,decorations.pathmorphing,backgrounds,positioning,fit,petri}

\usepackage{subfigure}
\usepackage{float}

\theoremstyle{plain}
\newtheorem{proposition}[equation]{Proposition}

\newtheorem{theorem}[equation]{Theorem}

\newtheorem{corollary}[equation]{Corollary}
\newtheorem{lemma}[equation]{Lemma}

\theoremstyle{definition}

\newtheorem{definition}[equation]{Definition}
\newtheorem{remark}[equation]{Remark}

\newtheorem{question}[equation]{Question} 
\newtheorem{fact}[equation]{Fact} 
\newcommand{\SA}{\mathbb{S}}

\newcommand{\E}{\mathcal{E}}
\newcommand{\F}{\mathcal{F}}
\newcommand{\B}{\mathcal{B}}

\newcommand{\FE}{\mathcal{F}}
\newcommand{\RE}{\mathcal{R}}

\newcommand{\SE}{\mathcal{S}}
\newcommand{\IE}{\operatorname{Ind}}

\newcommand{\csbe}{\mathcal{C}\mathcal{S}\mathcal{B}\mathcal{E}}

\newcommand{\kk}{\Bbbk}
\newcommand{\Ind}{\operatorname{Ind}}

\newcommand{\cd}{\operatorname{cochord}}
\newcommand{\reg}{\operatorname{reg}}

\newcommand{\im}{\operatorname{im}}

\newcommand{\projdim}{\operatorname{pd}}

\newcommand{\bp}{\operatorname{bp}}

\newcommand{\Iso}{\operatorname{Iso}}

\newcommand{\D}{\Delta}

\begin{document}
\title[Biclique vertex partitions and the regularity of $1$-subdivision graphs]{Chordal bipartite graphs, biclique vertex partitions and Castelnuovo-Mumford regularity of $1$-subdivision graphs}

\author{Yusuf Civan, Zakir Deniz, Oleg Duginov and Mehmet Akif Yetim}
\address{Department of Mathematics, S{\"{u}}leyman Dem{\.{\.i}}rel University,
Isparta, 32260, T{\"{u}}rk{\.{\.i}}ye.}
\email{yusufcivan@sdu.edu.tr}

\address{Department of Mathematics, D{\"{u}}zce University, D{\"{u}}zce, 81620, T{\"{u}}rk{\.{\.i}}ye.}
\email{zakirdeniz@duzce.edu.tr}

\address{Department of Informatics, Belarusian State University of Informatics and Radioelectronics, Minsk, 220005, Belarus.}
\email{oduginov@gmail.com}

\address{Department of Mathematics, S{\"{u}}leyman Dem{\.{\.i}}rel University,
Isparta, 32260, T{\"{u}}rk{\.{\.i}}ye.}
\email{akifyetim@sdu.edu.tr}


\date{\today}

\thanks{}


\begin{abstract}
A biclique in a graph $G$ is a complete bipartite subgraph (not necessarily induced), and the least positive integer $k$ for which the vertex set of $G$ can be partitioned into at most $k$ bicliques is the biclique vertex partition number $\bp(G)$ of $G$. We prove that the inequality $\reg(S(G))\geq |G|-\bp(G)$ holds for every graph $G$, where $S(G)$ is the $1$-subdivision graph of $G$ and $\reg(S(G))$ denotes the (Castelnuovo-Mumford) regularity of the graph $S(G)$. In particular, we show that the equality $\reg(S(B))=|B|-\bp(B)$ holds provided that $B$ is a chordal bipartite graph. Furthermore, for every chordal bipartite graph $B$, we prove that the independence complex of $S(B)$ is either contractible or homotopy equivalent to a sphere, and provide a polynomial time checkable criteria for when it is contractible, and describe the dimension of the sphere when it is not.  
\end{abstract} 

\maketitle

%%%%%%%%%%%%%%%%%%%%%%%%%%%%%%%%%%%%%%%%%%%%%%%%%%%%%%%%%%%%
\section{Introduction}\label{sect:intro}
In recent years, topology of simplicial complexes associated to graphs plays a crucial role on either bounding or even exact determination of various graph’s invariants. Such interrelations have also proven useful for uncovering topological/algebraic properties of simplicial complexes~\citep{kozlov, villa}. Our present work has a similar spirit. 

There has been much interest in graph covering and partitioning problems. Such problems are among the most fundamental and central subjects in graph theory due to their role in many mathematical models for various real-world applications. Although the edge/vertex covering/partitioning  of a host graph by cliques has a long history,  analogous problems by means of bicliques have also attracted a significant attention (see~\citep{FMPS} and the references therein). Recall that a \emph{biclique} in a graph $G$ is a complete bipartite subgraph (not necessarily induced), and the least positive integer $k$ for which the vertex set of $G$ can be partitioned into at most $k$ bicliques is the biclique vertex partition number $\bp(G)$ of $G$ (see Section~\ref{sect:prelim} for details). The decision problem related to the biclique vertex partition number is already known to be hard, even on several subclasses of bipartite graphs~\citep{ODug}.

One of the main aims of the current work is to show that the biclique vertex partition number can be bounded in terms of an algebraic invariant associated to a derived graph from the given one. Thus our method naturally uses tools from topological combinatorics/commutative algebra, and depends on the notion of (Castelnuovo-Mumford) regularity of $1$-subdivision graphs. Recall that for a given graph $G$, its $1$-\emph{subdivision graph} $S(G)$ is the graph obtained from $G$ by replacing every edge in $G$ by a path of length two. On the other hand, the (Castelnuovo-Mumford) \emph{regularity} $\reg(H)$ of a graph $H$ is the minimum integer $h$ for which the induced subcomplexes of the independence complex of $H$ have trivial homology groups in dimension $h$ or greater.

\begin{theorem}\label{thm:main-0}
The inequality $\reg(S(G))\geq |G|-\bp(G)$ holds for every graph $G$.
\end{theorem}

Theorem~\ref{thm:main-0} enables us to construct a lower bound to the biclique vertex partition number in terms of a particular domination parameter of the underlying graph. Recall that a subset $A$ of vertices is said to \emph{dominate} a set $B\subseteq V$ if $B\subseteq \bigcup_{a\in A}N_G[a]$. The minimum size of a set of vertices dominating a set $B$ is denoted by $\gamma(B,G)$. In particular, $\gamma(G):=\gamma(V,G)$ is the \emph{domination number} of $G$. On the other hand, the \emph{independence domination number}
$\gamma^i(G)$ is defined to be the maximum of $\gamma(I,G)$ over all independent sets $I$ in $G$~\citep{aharoni}.

\begin{corollary}\label{cor:main}
The inequalities $\gamma^i(G)\leq \bp(G)\leq \gamma(G)$ hold for every graph $G$.
\end{corollary}
We remark that the last inequality of Corollary~\ref{cor:main} is rather trivial and it is well-known~\citep{ODug}; hence, the interesting part lies in the first. The proof provided here uses the regularity of $1$-subdivision graphs as an intermediate invariant (see Theorem~\ref{thm:sg-tau}).\medskip 

We next  prove that the inequality of Theorem~\ref{thm:main-0} turns into an equality on the family of chordal bipartite graphs.
\begin{theorem}\label{thm:main-1}
$\reg(S(B))=|B|-\bp(B)$ holds for every chordal bipartite graph $B$.
\end{theorem}

We suspect that the class of chordal bipartite graphs is not the unique bipartite class over which the equality of Theorem~\ref{thm:main-1} holds. We discuss a possible superclass in Section~\ref{sect:concl} in which such subfamilies may be further identified.

In general, topological methods determining the homotopy types of independence complexes of graphs are limited. Their efficiencies generally depend on the existence of special structures over the underlying graphs. Such a calculation would be rather accessible if the graph lies in a hereditary class carrying some edge/vertex elimination schemes (see~\cite{kozlov}). For instance, the determination of the homotopy types of the independence complexes of chordal bipartite graphs themselves benefits from the existence of such a scheme~\citep{yetim}. Even though performing a subdivision on a graph may destroy such a structure, we next prove that there is still a chance that its existence may help us to describe the homotopy type after $1$-subdivision on chordal bipartite graphs.     

\begin{theorem}\label{thm:main-2}
For every chordal bipartite graph $B$, the independence complex of $S(B)$ is either contractible or homotopy equivalent to a sphere. In particular, there exists a polynomial time algorithm to decide when $\Ind(S(B))$ is contractible and determine the dimension of the sphere otherwise.
\end{theorem}

%%%%%%%%%%%%%%%%%%%%%%%%%%%%%%%%%%%%%%%%%%%
\section{Preliminaries}\label{sect:prelim}
We first recall some general notions and notations needed throughout the
paper, and repeat some of the definitions mentioned in the introduction more
formally.

\subsection{Graphs}
All graphs we consider are finite and simple, i.e., do not have any loops or multiple edges. By writing $V(G)$ and $E(G)$, we mean the vertex set and the edge set of $G$, respectively. An edge between $u$ and $v$ is denoted by $e= uv$ or $e=\{u,v\}$ interchangeably. If $U\subset V$, the graph induced on $U$ is written $G[U]$, and in particular, we abbreviate $G[V\backslash U]$ to $G-U$, and write $G-x$ whenever $U=\{x\}$.
The \emph{complement} of $G$, denoted by $\overline{G}$, is the graph with the same vertex set $V$ and such that $uv\in E(\overline{G})$ if and only if $uv\notin E(G)$. The open and closed neighborhoods of a vertex $v$ are $N_G(v)=\{u\in V(G): \; uv\in E(G) \}$ and $N_G[v]=N_G(v)\cup \{v \}$, respectively. The size of the set $N_G(v)$ is called the \emph{degree} of $v$ in $G$ and denoted by $\deg_G(v)$, and the maximum and the minimum degrees of a graph G are denoted by $\D(G)$ and $\delta(G)$ respectively.
The closed neighborhood of an edge $e=uv$ is defined to be the set $N_G[e]:=N_G[u]\cup N_G[v]$. 

Throughout the paper, $K_n$, $P_n$ and $C_k$ will denote the complete, path and cycle graphs on $n\geq 1$ and $k\geq 3$ vertices, respectively. Moreover, we denote by $K_{p,q}$, the complete bipartite graph for any $p, q\geq 1$. In particular, the complete bipartite graph $K_{1,q}$ is called a \emph{star}, and the graph $K_{1,3}$ is known as the \emph{claw}. 

If $\FE$ is a family of graphs, we say that a graph $G$ is $\FE$-free if no induced subgraph of $G$ is isomorphic to a graph from the family $\FE$. A graph $G$ is called \emph{chordal} if it does not contain an induced cycle of length greater than three. Moreover, a graph $G$ is said to be co-chordal if its complement $\overline{G}$ is a chordal graph. A collection $\mathcal{G}=\{H_1, H_2,\ldots, H_r\}$ of subgraphs of $G$ is called a \emph{co-chordal cover} of $G$ (of length $r$), if each edge of $G$ belongs to at least one of the subgraphs in $\mathcal{G}$. The \emph{co-chordal cover number} of $G$, denoted by $\cd(G)$, is the minimum number $k$ such that $G$ admits a co-chordal cover of length $k$. 
 
An \emph{independent set} in a graph is a set of pairwise non-adjacent vertices, while a \emph{clique} means a set of pairwise adjacent vertices. A graph is \emph{bipartite} if its vertex set can be partitioned into two independent sets. A graph is said to be \emph{chordal bipartite} if it is bipartite and does not contain an induced cycle of length greater than four. 
A \emph{biclique} in a graph is a (not necessarily induced) subgraph isomorphic to a complete bipartite graph. A set $\RE=\{R_1,R_2,\ldots,R_k\}$ of bicliques of a graph $G$ is a \emph{biclique vertex cover} of $G$, if each vertex of $G$ belongs to at least one biclique in $\RE$. A biclique vertex cover $\RE=\{R_1,R_2,\ldots,R_k\}$ is said to be a \emph{biclique vertex partition} if bicliques in $\RE$ are pairwise disjoint, that is, each vertex of $G$ belongs to exactly one biclique in $\RE$. The \emph{biclique vertex partition number} $\bp(G)$ of a graph $G$ is defined to be the least integer $k$ such that $G$ admits a biclique vertex partition of size $k$. 

\begin{remark} It is known from \citep{FMPS} that a graph has a biclique vertex cover of size at most $k$ if and only if it has a biclique vertex partition of size $k$. Therefore we do not distinguish between the biclique vertex cover and the biclique vertex partition of a graph $G$ and use the notation $\bp(G)$. 
\end{remark}

A \emph{matching} of a graph $G$ is a subset of edges, no two of which share a common vertex. When $M$ is a matching in a graph $G$, we denote by
$V(M)$, the set of vertices incident to edges in $M$.
An induced matching is a matching $M$ such that no two vertices belonging to different edges of $M$ are adjacent. The maximum size $\im(G)$ of an induced matching of $G$ is known as the \emph{induced matching number} of $G$.

We write bipartite graphs in the form $B=(X,Y;E)$ in which the vertex set $V(B)$ has the corresponding bipartition $X\cup Y$. 
A vertex $x$ in $B$ is said to be a \emph{simple vertex} if for any $u,v\in N_B(x)$, either $N_B(u)\subseteq N_B(v)$ or $N_B(v)\subseteq N_B(u)$ holds~\citep[pp. $77$]{BLS-gclass}. This clearly defines a linear order among neighborhoods of neighbors of a simple vertex. An edge $e$ of $B$ is called a \emph{bisimplicial edge} if $N_B[e]$ induces a biclique in $B$.

\begin{proposition}\textnormal{\cite{HMP}}\label{prop:simple}
Every chordal bipartite graph has at least one simple vertex in each partite set. 
\end{proposition}

The following fact is immediate from Proposition~\ref{prop:simple}.

\begin{fact}\label{fact:bisimplicial}
Every chordal bipartite graph having at least one edge has a bisimplicial edge. 
\end{fact}

Indeed, let $v$ be a simple vertex of a chordal bipartite graph $B$ and $u_1, u_2,\ldots, u_k$ be its neighbors such that $N_B(u_i)\subseteq N_B(u_j)$ for $1\leq i<j\leq k$. Then observe that the edge $vu_1$ is a bisimplicial edge. In the sequel, we often appeal to the existence of such bisimplicial edges in chordal bipartite graphs.

\subsection{Simplicial Complexes}

An \emph{(abstract) simplicial complex} $X$ on a finite set $V$ is a family of subsets of $V$ closed under inclusion and $\{v\}\in X$ for every $v\in V$.

The elements of $V$ and $X$ are called \textit{vertices} and \textit{faces} of $X$, respectively. The dimension of a face $A \in X$ is $\dim(A):=|A|-1$. The dimension of $X$ is $\dim(X):=\max\{\dim(A)\colon A\in X\}$. The \textit{(simplicial) join} of two complexes $X_1$ and $X_2$ is defined by
\begin{displaymath}
X_1 \ast X_2:=\{A_1 \cup A_2\colon  A_i\in X_i, i=1,2\}.
\end{displaymath}

In particular, the join of a simplicial complex $X$ and the zero-dimensional sphere $\SA^0=\{ \emptyset, \{a\},\{b\} \}$ is called the \emph{suspension} of $X$, that is, $\Sigma X:= \SA^0 \ast X$. Similarly, the join of $X$ and the simplicial complex $\{\emptyset,\{a\}\}$ is called the cone of $X$ with apex $a$.

One can associate to a graph $G$, the simplicial complex $\Ind(G)$, the independence complex of $G$, whose faces correspond to independent sets of $G$.
 
\begin{definition}
Let $X$ be a simplicial complex on the vertex set $V$, and let $\Bbbk$ be any field. Then the \textit{Castelnuovo-Mumford regularity} $\reg_{\Bbbk}(X)$ of $X$ over $\Bbbk$ is defined by
\begin{displaymath}
\reg_{\Bbbk}(X):=\max \{j\colon \widetilde{H}_{j-1}(X[S]; \Bbbk) \neq 0 \; \text{for some } S\subseteq V  \} 
\end{displaymath}
where $X[S]:=\{ F\in X\colon F\subseteq S\}$ is the induced subcomplex of $X$ by $S$, and $\widetilde{H}_{\ast}(-,\Bbbk)$ denotes the reduced singular homology. 
\end{definition}

The regularity of a graph $G$ (over $\Bbbk$) is defined to be the regularity of its independence complex $\Ind(G)$. We write $\reg(G)$ instead of $\reg(\Ind(G))$, and note that our results are independent of the characteristic of the coefficient field, so we drop $\kk$ from our notation.

The following provides an inductive bound on the regularity of graphs. 

\begin{corollary}\textnormal{\cite{DHS-boundpd}}\label{cor:induction-sc}
Let $G$ be a graph and let $v\in V$ be given. Then
\begin{equation*}
\reg (G)\leq \max\{\reg(G-v), \reg(G-N_G[v])+1\}.
\end{equation*}
Moreover, $\reg(G)$ always equals to one of $\reg(G-v)$ or $\reg(G-N_G[v])+1$. 
\end{corollary}

\begin{lemma}\textup{\cite{BC-prime}}\label{lem:deg-one}
If $N_G(x)=\{y\}$ in $G$, then either $\reg(G)=\reg(G-x)$ or else
$\reg(G)=\reg(G-N_G[y])+1$.
\end{lemma}

\begin{lemma}\textnormal{\citep{MK-im,russ}}\label{lem:cochord}
$\im(G)\leq \reg(G)\leq \cd(G)$ for every graph $G$.
\end{lemma}

The existence of vertices satisfying certain extra properties is useful when dealing with the homotopy type of simplicial complexes. Recall that when $(X,A)$ is a CW pair, then $A$ is said to be \emph{contractible in} $X$, if the inclusion map $A\hookrightarrow X$ is homotopic to a constant map~\citep[Example $0.14$]{AH-top}. Note that a simplicial complex which is a cone of another is contractible. In particular, if a graph $G$ has an isolated vertex $x$, then $\IE(G)$ is contractible, since $\IE(G)$ is the cone of $\IE(G-x)$ with apex $x$.

\begin{theorem}\label{thm:hom-simp-induction}
If $\Ind(G-N_G[x])$ is contractible in $\Ind(G-x)$, then $\Ind(G)\simeq \Ind(G-x)\vee \Sigma(\Ind(G-N_G[x])$.
\end{theorem}

In the case of the independence complexes of graphs, the existence of a (closed or open) dominated vertex may guarantee that the required condition of Theorem~\ref{thm:hom-simp-induction} holds. 

\begin{corollary}\textnormal{\cite{AE}}\label{cor:hom-induction}
If $N_G(u)\subseteq N_G(v)$, then there is a homotopy equivalence $\IE(G)\simeq \IE(G-v)$. 
\end{corollary}
%%%%%%%%%%%%%%%%%%%%%%%%%%%%%%%%%%%%%%%%%%%%%%%

%%%%%%%%%%%%%%%%%%%%%%%%%%%%%%%%%%%%%%%%%%%%%%%%%%%%%%
\section{Proofs of Theorems~\ref{thm:main-0} and~\ref{thm:main-1} }

In this section, we provide one of our main results, and show that the biclique vertex partition number of a chordal bipartite graph can be stated in terms of the Castelnuovo-Mumford regularity of its $1$-subdivision graph. In order to do that, we first bring forward some preliminary results from topological combinatorics.

Let $n,m$ be two positive integers. We write the vertex set of the complete bipartite graph $K_{n,m}$ as $V(K_{n,m})=X\cup Y$, where
$X=\{x_1,\ldots,x_n\}$ and $Y=\{y_1,\ldots,y_m\}$. For every pair $(i,j)\in [n]\times [m]$, we denote by $e_{ij}$, the edge between $x_i$ and $y_j$ as well as the corresponding vertex in the $1$-subdivision graph $S(K_{n,m})$. The \emph{whisker} $W_A(G)$ of a graph $G$ at a vertex subset $A\subseteq V(G)$ is the graph obtained from $G$ by attaching a pendant edge to each vertex of $A$ in $G$. In other words, $W_A(G)$ is defined to be the graph on $V(W_A(G)):=V(G)\cup \{a'\colon a\in A\}$ such that $E(W_A(G)):=E(G)\cup \{aa'\colon a\in A\}$.

\begin{proposition}\label{prop:whisker-comp}
The independence complex of the graph $W_Y(S(K_{n,m}))$ is contractible. \end{proposition}
\begin{proof}
For simplicity, we write $H_{n,m}=W_Y(S(K_{n,m}))$. For each vertex $y_i\in Y$, let $y^*_i$ denote the whisker at $y_i$. Since $N_{H_{n,m}}(y^*_i)\subseteq N_{H_{n,m}}(e_{1i})$, the complexes $\Ind(H_{n,m})$ and $\Ind(H_{n,m}-e_{1i})$ are homotopy equivalent by Corollary~\ref{cor:hom-induction}. This means that we may remove each vertex $e_{1j}$ from $H_{n,m}$ for every $j\in [m]$ without altering the homotopy type of $\Ind(H_{n,m})$. It then follows that $\Ind(H_{n,m})\simeq \Ind(H_{n,m}-\{e_{1j}\colon 1\leq j\leq m\})$. However, the vertex $x_1$ becomes isolated in the graph $H_{n,m}-\{e_{1j}\colon 1\leq j\leq m\}$. Therefore, the complex
$\Ind(H_{n,m}-\{e_{1j}\colon 1\leq j\leq m\})$ is contractible, so
is $\Ind(H_{n,m})$.
\end{proof}

\begin{proposition}\label{prop:subd-comp}
The independence complex of the graph $S(K_{n,m})$ is homotopy equivalent to a sphere of dimension $(n+m-2)$.
\end{proposition}
\begin{proof}
We proceed by the induction $n$. If $n=1$, then $\Ind(S(K_{1,m}))\simeq \Ind(mK_2)\simeq \SA^{m-1}$. So, the claim holds.  
Assume next that $n>1$, and  pick a vertex $x_1\in X$. Observe that $S(K_{n,m})-N_{S(K_{n,m})}[x_1]\cong S(K_{n,m}-x_1)$. Since $S(K_{n,m}-x_1)\cong S(K_{n-1,m})$, the independence complex of the latter graph is homotopy equivalent to the sphere of dimension $(n+m-3)$ by the induction hypothesis. 

On the other hand, since $S(K_{n,m})-x_1$ is isomorphic to the graph $W_Y(S(K_{n-1,m}))$, the deletion subcomplex of the vertex $x_1$ is contractible. It then follows from Theorem~\ref{thm:hom-simp-induction} that $\Ind(S(K_{n,m}))\simeq \Sigma(\Ind(S(K_{n,m})-N_{S(K_{n,m})}[x_1]))\simeq \Sigma(\SA^{n+m-3})\simeq \SA^{n+m-2}$.
\end{proof}

\begin{lemma}\label{lemma:compbipreg}
$\reg(S(K_{n,m}))=n+m-1$ for every pair $n,m$ of positive integers.
\end{lemma}
\begin{proof}
Firstly, we have $\reg(S(K_{n,m}))\geq n+m-1$ by Proposition~\ref{prop:subd-comp}. To complete the proof, we provide a co-chordal covering of $S(K_{n,m})$ of size $(n+m-1)$ so that $\reg(S(K_{n,m}))\leq \cd(S(K_{n,m}))\leq n+m-1$ by Lemma~\ref{lem:cochord}. For each $j\in [m]$, let
$A_j$ be the subgraph of $S(K_{n,m})$ consisting of the star centered at $y_j$ together with the additional edge $(e_{1j},x_1)$. Similarly, for each $2\leq i\leq n$, let $B_i$ be the star in $S(K_{n,m})$ centered at the vertex $x_i$. Now, the family $\{A_j,B_i\colon 1\leq j\leq m\;\textnormal{and}\; 2\leq i\leq n\}$ is a co-chordal covering of $S(K_{n,m})$. This completes the proof.
\end{proof}

\begin{proof}[{\bf Proof of Theorem~\ref{thm:main-0}}]
Suppose that $\bp(G)=k$ and let $\B:=\{B_1,\ldots,B_k\}$ be a biclique vertex partition of $G$. We fix a bipartition $V(B_i)=X_i\cup Y_i$, and set $T_i:=\{xy\in E(B_i)\colon x\in X_i, y\in Y_i\}$ for each $1\leq i\leq k$. If we define  
\begin{displaymath}
H_i:=S(G)[X_i\cup Y_i\cup T_i]\cong S(K_{|X_i|,|Y_i|}),
\end{displaymath}
we have that $\reg(H_i)=|X_i|+|Y_i|-1$ by Lemma~\ref{lemma:compbipreg}.
Observe further that there exists no edges between vertices in $H_i$ and $H_j$ for any $i\neq j$. Moreover, the graph $H(G):=\bigcup\limits_{i=1}^k H_i$ is an induced subgraph of $S(G)$ so that $\reg(S(G))\geq \reg(H(G))=\sum_{i=1}^n\reg(H_i)=|G|-\bp(G)$ as claimed.
\end{proof}

We remark that the gap between $\reg(S(G))$ and $|G|-\bp(G)$ could be arbitrarily large even for bipartite graphs. Indeed, let $p\geq 3$ and $n\geq 2$ be  integers, and denote by $R^p_n$, the graph obtained from the disjoint union of $n$ copies of $C_{p}$ together with a star $K_{1,n}$ in which each degree one vertex $x_i$ of $K_{1,n}$ is adjacent to a unique vertex of the $i^{\textnormal{th}}$-copy of $C_{p}$ for every $1\leq i\leq n$ so that $\deg_{R^p_n}(x_i)=2$ (see Figure~\ref{fig:k14-extend} for the illustration of the graph $R^4_4$). For $p=10$, we have that
\begin{displaymath}
|R^{10}_n|-\bp(R^{10}_n)=(11n+1)-4n=(7n+1)<8n=\reg(S(nC_{10})\cup S(K_{1,n}))\leq \reg(S(R^{10}_n))
\end{displaymath} 
for every $n\geq 2$, where $R^{10}_n$ is clearly a bipartite graph.

However, we note that the lower bound of Theorem~\ref{thm:main-0} is stronger than the induced matching bound of Lemma~\ref{lem:cochord}. Indeed, it is not difficult to prove that $\im(S(G))=|G|-\gamma(G)$ for every graph $G$ without any isolated vertex. Therefore, we conclude that
\begin{displaymath}
\im(S(G))=|G|-\gamma(G)\leq |G|-\bp(G)\leq \reg(S(G)),
\end{displaymath}
where the first inequality could be strict for various graphs. For instance, $\gamma(R^4_n)=2n$ and $\bp(R^4_n)=n+1$ for every $n\geq 2$.

\begin{figure}[ht]
\centering 
\begin{tikzpicture}
    % Coordinates for the vertices of K_{1,4}
\node [nod2] at (0,0) (X) []{};
\node [nod2] at (0.5,0.5) (X1) []{};
\node [nod2] at (-0.5,0.5) (X2) []{};
\node [nod2] at (-0.5,-0.5) (X3) []{};
\node [nod2] at (0.5,-0.5) (X4) []{};
\node [nod2] at (1,1) (A) []{};
\node [nod2] at (-1,1) (B) []{};
\node [nod2] at (-1,-1) (C) []{};
\node [nod2] at (1,-1) (D) []{};
   
    % Edges of K_{1,n}
    \draw (X) -- (X1);
    \draw (X) -- (X2);
    \draw (X) -- (X3);
    \draw (X) -- (X4);
    \draw (A) -- (X1);
    \draw (B) -- (X2);
    \draw (C) -- (X3);
    \draw (D) -- (X4);
    % C4 at A   
    \node [nod2] at (1.8,1) (A1) []{};
    \node [nod2] at (1,1.8) (A2) []{};
    \node [nod2] at (1.8,1.8) (A3) []{};
    \draw (A) -- (A1) -- (A3) -- (A2) -- (A);
    
    % C4 at B
    \node [nod2] at (-1.8,1) (B1) []{};
    \node [nod2] at (-1,1.8) (B2) []{};
    \node [nod2] at (-1.8,1.8) (B3) []{};
    \draw (B) -- (B1) -- (B3) -- (B2) -- (B);
    
    % C4 at C
    \node [nod2] at (-1.8,-1) (C1) []{};
    \node [nod2] at (-1,-1.8) (C2) []{};
    \node [nod2] at (-1.8,-1.8) (C3) []{};
    \draw (C) -- (C1) -- (C3) -- (C2) -- (C);
    
    % C4 at D
    \node [nod2] at (1.8,-1) (D1) []{};
    \node [nod2] at (1,-1.8) (D2) []{};
    \node [nod2] at (1.8,-1.8) (D3) []{};
    \draw (D) -- (D1) -- (D3) -- (D2) -- (D);
    
\end{tikzpicture}
\caption{The graph $R^4_4$.}
\label{fig:k14-extend}
\end{figure}
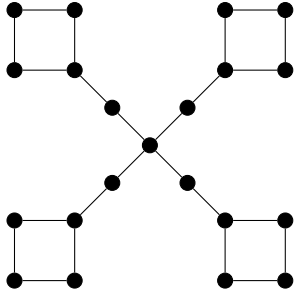

For a given graph $G$, we denote by $G^\circ$, the set of all non-isolated vertices in $G$.

\begin{theorem}\label{thm:sg-tau}
$\reg(S(G))\leq |G|-\gamma^i(G)$ for any graph $G$.
\end{theorem}
\begin{proof}
We proceed by an induction on the order of $G$. 
Note that isolated vertices (when exist) contribute equally to both the independence domination number and the order of the graph $G$, while having no effect to the regularity. We may therefore disregard their existence throughout the proof.

Suppose that $\gamma^i(G)=\gamma(A,G^{\circ})=|X|$ for some independent set $A$ of $G^{\circ}$ such that $A\subseteq N_G(X)$. We can always pick a vertex $x\in X$ such that $x\notin A$, since $A$ is independent and $X$ is minimal. We then apply the inequality of Corollary~\ref{cor:induction-sc} to $S(G)$ at the vertex $x$. It follows that $\reg(S(G))\leq \max\{\reg(S(G)-x),\reg(S(G)-N_{S(G)}[x])+1\}$. Observe that $S(G)-N_{S(G)}[x]\cong S(G-x)$. Furthermore, if we let $N_G(x)=\{x_1,\ldots,x_n\}$ and define $H_x:=S(G)-x$, the vertex $e_i:=xx_i$ of $H_x$ is of degree one in $H_x$ for each $i\in [n]$. If we apply Lemma~\ref{lem:deg-one} at these vertices consecutively , it then follows that there exists a subset $B\subseteq N_G(x)$ such that $\reg(H_x)=\reg(S(G-(B\cup\{x\})))+k$, where $|B|=k$. Therefore, we have either $\reg(S(G))=\reg(S(G-x))+1$ or else $\reg(S(G))=\reg(S(G-(B\cup \{x\})))+k$ by Corollary~\ref{cor:induction-sc}.

Assume first that $\reg(S(G))=\reg(S(G-x))+1$. If no vertex of $A$ is isolated in $G-x$,  any subset of $G-x$  dominating $A$ in $G-x$ also dominates $A$ in $G$ so that $\gamma^i(G)=\gamma(A,G^{\circ})\leq \gamma(A,(G-x)^{\circ})\leq \gamma^i(G-x)$. Then by the induction, we have $\reg(S(G))=\reg(S(G-x))+1\leq (|G|-1)-\gamma^i(G-x)+1\leq |G|-\gamma^i(G)$. If $C$ is a subset of $A$ consisting of those vertices that are isolated in $G-x$, then $\gamma(A-C,(G-x)^{\circ})\geq \gamma(A,G^{\circ})-1$ so that $\gamma^i(G-x)\geq \gamma^i(G)-1$.
Therefore, we have $\reg(S(G))=\reg(S(G-x))+1\leq (|G|-|C|-1)-\gamma^i(G-x)+1\leq (|G|-2)-(\gamma^i(G)-1)+1=|G|-\gamma^i(G)$.

Suppose next that $\reg(S(G))=\reg(S(G-(B\cup \{x\})))+k$. Observe first that if $k=0$, then $B=\emptyset$. In such a case,  since $\reg(S(G))=\reg(S(G-x))$,  we conclude the claim by a similar argument as  above. If $k>0$, assume that $D$ is a subset of $A$ such that any vertex in $D$ is isolated in $G-(B\cup \{x\})$, and let $\gamma(A-D,(G-(B\cup\{x\}))^{\circ})=|Y|$. Pick a neighbor for each vertex of $D$ in $G$, and let $Z$ be the set of these neighbors so that $|Z|\leq |D|$. Observe that $A\subseteq N_G(Y\cup Z\cup \{x\})$. Therefore, we have $\gamma(A-D,(G-(B\cup \{x\}))^{\circ})=|Y|\geq |X|-|Z|-1\geq \gamma(A,G^{\circ})-|D|-1$. It then follows that $\gamma^i(G-(B\cup \{x\}))\geq \gamma^i(G)-|D|-1$. Now, we apply the induction to conclude that 
\begin{align*}
\reg(S(G))&=\reg(S(G-B\cup \{x\}))+k\\
&\leq (|G|-|D|-|B|-1)-\gamma^i(G-B\cup \{x\})+k\\
&=|G|-\gamma^i(G).
\end{align*}
\end{proof}

Now, the combination of Theorems~\ref{thm:main-0} and~\ref{thm:sg-tau} shows that $\gamma^i(G)\leq \bp(G)$ for every graph $G$. That completes the proof of Corollary~\ref{cor:main}.\medskip

Observe that both inequalities of Corollary~\ref{cor:main} could be strict even for  chordal bipartite graphs. If we consider the graph $B$ depicted in Figure~\ref{fig:indom-bp}, we have that 
$\gamma^i(B)=2<3=\bp(B)=\gamma(B)$. Next, if we denote by $B\cup C_4$, the bipartite graph obtained from the disjoint union of the graph $B$ and a $4$-cycle, it follows that $3=\gamma^i(B\cup C_4)<4=\bp(B\cup C_4)<5=\gamma(B\cup C_4)$.

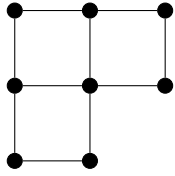
\begin{figure}[ht]
\centering     %%% not \center
\begin{tikzpicture}
\node [nod2] at (0,0) (v1) []{};
\node [nod2] at (1,0) (v2) []{}  
       edge [] (v1);      
\node [nod2] at (2,0) (v3) []{} 
       edge [] (v2);
\node [nod2] at (0,-1) (v4) []{}  
	   edge [] (v1);
\node [nod2] at (1,-1) (v5) []{} 
	   edge [] (v2)
	   edge [] (v4);
\node [nod2] at (2,-1) (v6) []{} 
	   edge [] (v3)
	   edge [] (v5);
\node [nod2] at (0,-2) (v7) []{} 
	   edge [] (v4);
\node [nod2] at (1,-2) (v8) []{} 
	   edge [] (v5)
	   edge [] (v7);
\end{tikzpicture}
\caption{A chordal bipartite graph $B$ such that $\gamma^i(B)=2< 3=\bp(B)$.}
\label{fig:indom-bp}
\end{figure}

We are now ready to complete the proof of Theorem~\ref{thm:main-1}, while we need a technical result first.

\begin{lemma}\label{lem:bp-induct}
If $u,v\in V(G)$ are two vertices with $N_G(u)\setminus \{v\}\subseteq N_G(v)$, then $\bp(G)\leq \bp(G-v)$.
\end{lemma}
\begin{proof}
Suppose that $\B=\{B_1,\ldots,B_n\}$ is a biclique vertex partition of $G-v$. We may assume without loss of generality that $u\in B_1$. If $V(B_1)=U_1\cup U_2$ such that $u\in U_1$, we then define a new biclique $B'_1$ with $V(B'_1):=(U_1\cup \{v\})\cup U_2$ and set $\B':=\{B'_1,B_2,\ldots,B_n\}$. Obviously, $\B'$ is a biclique vertex partition of $G$. 
\end{proof}

\begin{proof}[{\bf Proof of Theorem~\ref{thm:main-1}}]
Since $\reg(S(B))\geq |B|-\bp(B)$ by Theorem~\ref{thm:main-0}, we are left to prove that $\reg(S(B))\leq |B|-\bp(B)$.

Now, let $x,y\in V(B)$ be given such that $N_B(x)\subseteq N_B(y)$. We then apply the inequality of Corollary~\ref{cor:induction-sc} to $S(B)$ at the vertex $y$. It follows that 
$$\reg(S(B))\leq \max\{\reg(S(B)-y),\reg(S(B)-N_{S(B)}[y])+1\}.$$ 
Observe that $S(B)-N_{S(B)}[y]\cong S(B-y)$. Furthermore, if we let $N_B(y)=\{x_1,\ldots,x_n\}$ and define $H_y:=S(B)-y$, the vertex $e_i:=yx_i$ of $H_y$ is of degree one in $H_y$ for each $i\in [n]$. If we apply Lemma~\ref{lem:deg-one} at these vertices consecutively, it then follows that there exists a subset
$L\subseteq N_B(y)$ such that $\reg(H_y)=\reg(S(B-(L\cup\{y\})))+k$, where $|L|=k$. Therefore, we have either $\reg(S(B))=\reg(S(B-y))+1$ or else $\reg(S(B))=\reg(S(B-(L\cup \{y\})))+k$ by Corollary~\ref{cor:induction-sc}.

Assume first that $\reg(S(B))=\reg(S(B-y))+1$. It then follows from the induction that 
$$\reg(S(B))=\reg(S(B-y))+1\leq |B-y|-\bp(B-y)+1\leq |B|-\bp(B),$$
where the last inequality is due to Lemma~\ref{lem:bp-induct}.

Suppose next that $\reg(S(B))=\reg(S(B-(L\cup \{y\})))+k$. Observe first that if $k=0$, then $L=\emptyset$. In such a case,  we have 
$$\reg(S(B))=\reg(S(B-y))\leq |B-y|-\bp(B-y)\leq |B|-(\bp(B-y)+1)\leq |B|-\bp(B).$$ 

Finally, we consider the case $k>0$. If $\B$ is a biclique vertex partition of $B-(L\cup \{y\})$, then $\B\cup \{K_{1,k}\}$ is a biclique vertex partition for $B$, where $V(K_{1,k})=\{y\}\cup L$. Thus, we have $\bp(B-(L\cup \{y\}))+1\geq \bp(B)$. It then follows from the induction hypothesis that

\begin{align*}
\reg(S(B))=&\reg(S(B-(L\cup \{y\})))+k\\
&\leq |(B-(L\cup \{y\})|-\bp(B-(L\cup \{y\}))+k\\
&=|B|-|L|-1-\bp(B-(L\cup \{y\}))+k\\
&=|B|-(\bp(B-(L\cup \{y\}))+1)\\
&\leq |B|-\bp(B).
\end{align*}
\end{proof}
%%%%%%%%%%%%%%%%%%%%%%%%%%%%%%%%%%%%%%%%%%%%%%%
\section{Proof of Theorem~\ref{thm:main-2}}\label{sect:thmmanin-2}
In this section, we present the proof of Theorem~\ref{thm:main-2} in two steps. In fact, we prove that the claims of Theorem~\ref{thm:main-2} hold for a larger family of bipartite graphs containing chordal bipartite graphs as a subfamily. For that purpose, we introduce a specific elimination scheme for bipartite graphs. This procedure leads to a polynomial time algorithm for determining the homotopy type of $1$-subdivision graphs of bipartite graphs in the aforementioned superclass.

Let $\{e_1,\ldots,e_k\}$ be an ordered subset of edges in $B$ with $k\geq 1$. Then for each $1\leq i\leq k$, we set $B_i:=B_{i-1}-N_{B_{i-1}}[e_i]$, where $B_0:=B$.

\begin{definition}
An ordered subset $\{e_1,\ldots,e_k\}$ of edges in a bipartite graph $B$ is said to be a \emph{biclique elimination sequence} for $B$, if
\begin{itemize}
\item $e_i$ is a bisimplicial edge in $B_{i-1}$ for each $1\leq i\leq k$,\\
\item $B_k$ is an edgeless graph.
\end{itemize}
Furthermore, a biclique elimination sequence $\{e_1,\ldots,e_k\}$ of a  bipartite graph $B$ is called \emph{simple} if at least one end vertex of $e_i$ is a simple vertex of $B_{i-1}$ for every $1\leq i\leq k$.  
\end{definition}

%%%%%%%%%%%%%%%%%%%%%%%%%%%%%%%%%%%%%%%%%%%%
\begin{theorem}\label{thm:main-2-a}
If a bipartite graph $B$ admits a simple biclique elimination sequence, the independence complex of $S(B)$ is either contractible or homotopy equivalent to a sphere.
\end{theorem}
\begin{proof}
We may initially assume without loss of generality that $B$ is connected. 
Let $\{e_1,\ldots,e_l\}$ be a simple biclique elimination sequence for
$B$. 

We consider the bisimplicial edge $e_1$, and let $s\in e_1$ be a simple vertex of $B$. We write $N_B(s)=\{s_1,\ldots,s_k\}$ such that $N_B(s_1)\subseteq \ldots \subseteq N_B(s_k)$, and set $N_B(s_1)=\{x_1,x_2,\ldots,x_n\}$, where $s=x_1$ so that $e_1=ss_1$.

For the graph $S(B)-s_k$, each vertex $x_i$ has a degree one neighbor, namely $f_i=s_kx_i$ in $S(B)$. Since $N_{S(B)-s_k}(f_i)\subseteq N_{S(B)-s_k}(s_1x_i)$, we may remove each vertex of the form $s_1x_i$ from $S(B)-s_k$ without altering the homotopy type by Corollary~\ref{cor:hom-induction}. However, in the resulting graph, the vertex $s_1$ becomes isolated; hence, $\Ind(S(B)-s_k)$ is contractible. Then we conclude that $\Ind(S(B))\simeq \Sigma(\Ind(S(B)-N_{S(B)}[s_k]))$  by Theorem~\ref{thm:hom-simp-induction}. Since $S(B)-N_{S(B)}[s_k]\cong S(B-s_k)$, we have $\Ind(S(B))\simeq \Sigma(\Ind(S(B-s_k)))$. If the removal of $s_k$ from $B$ causes a vertex to be isolated, then $\Ind(S(B))$ is contractible.

Now, the vertex $s$ is still a simple vertex in $B-s_k$, we may repeat the above argument one by one to conclude that $\Ind(S(B))\simeq \Sigma^{k-1}(\Ind(S(B-\{s_2,\ldots,s_k\})))$. Let us write $H_s:=S(B-\{s_2,\ldots,s_k\})$. The vertex $s$ has degree one in $H_s$, and $N_{H_s}(s)=\{ss_1\}\subseteq N_{H_s}(s_1)$ so that $\Ind(H_s)\simeq \Ind(H_s-s_1)$. Observe that each vertex $g_i=s_1x_i$ is of degree one in $H_s-s_1$, where $g_1=e_1$ by our earlier convention. This means that if $x_iy\in E$ for some $y\in V$ and $i>1$, we may remove every such  vertex $x_iy$ from $H_s-s_1$ without altering the homotopy type of $\Ind(H_s-s_1)$. In other words, we have $\Ind(H_s-s_1)\simeq \Ind(nK_2\cup S(B-N_B[g_1]))$, where $nK_2$ is the graph consisting of $n$ disjoint copies of $K_2$ whose corresponding edges are $(x_1,s_1x_1),\ldots,(x_n,s_1x_n)$. Once again, we note that if the removal of any of these vertices causes a vertex to be isolated, then $\Ind(H_s-s_1)$ is contractible, so is $\Ind(S(B))$.

{\em Case} $1$: $V(B-N_B[g_1])=\emptyset$. In such a case, we have that $H_s-s_1\cong nK_2$. It then follows that 
\begin{displaymath}
\Ind(S(B))\simeq \Sigma^{k-1}(\Ind(H_s-s_1))\simeq \Sigma^{k-1}(\Ind(nK_2))\simeq \Sigma^{k-1}(\SA^{n-1})\simeq \SA^{n+k-2}.
\end{displaymath}

{\em Case} $2$: $V(B-N_B[g_1])\neq \emptyset$. Since the graphs
$nK_2$ and $S(B-N_B[g_1])$ are disjoint, we have $\Ind(H_s-s_1)\simeq \Sigma^{n}(\Ind(S(B-N_B[g_1])))$. As result we conclude that

\begin{align*}
\Ind(S(B))\simeq \Sigma^{k-1}(\Ind(H_s-s_1))
&\simeq \Sigma^{k-1}(\Sigma^{n}(\Ind(S(B-N_B[g_1]))))\\
&\simeq \Sigma^{n+k-1}(\Ind(S(B-N_B[g_1]))).
\end{align*}
Finally, since $\{e_2,\ldots,e_l\}$ is a simple biclique elimination sequence for the graph $B-N_B[g_1]=B-N_B[e_1]$, the proof follows by the induction.
\end{proof}

Since every chordal bipartite graph has a simple vertex from Proposition~\ref{prop:simple} and the property of being a chordal bipartite is closed under taking induced subgraphs, we have the following.

\begin{corollary}\label{cor:chordalbip-bes}
Every chordal bipartite graph $B$ admits a simple biclique elimination sequence. In particular, $\Ind(S(B))$ is either contractible or homotopy equivalent to a sphere.
\end{corollary}

We next show that the parity of the homotopy type of the independence complex of $S(B)$ for a bipartite graph $B$ with a simple biclique elimination sequence can be decided in polynomial time. 

\begin{definition}
Let $e$ be a bisimplicial edge in $B$. We denote by $\Iso_B(e)$, the set of all isolated vertices in $B-N_B[e]$. We call a biclique elimination sequence $\E:=\{e_1,\ldots,e_k\}$,  \emph{complete} if $\Iso_B(\E):=\bigcup\limits_{i=1}^k \Iso_{B_{i-1}}(e_i)=\emptyset$. 

We say that a sequence $\E:=\{e_1,\ldots,e_k\}$, a $\csbe$-\emph{sequence} of a bipartite graph $B$ if it is a complete simple biclique elimination sequence in $B$. In particular, we call a bipartite graph a $\csbe$-\emph{graph} if it admits a $\csbe$-sequence.
\end{definition}

We remark that every bipartite graph is an induced subgraph of a bipartite $\csbe$-graph. Indeed, let $B$ be an arbitrary bipartite graph, and let $C_4^v$ be a copy of $C_4$ for each vertex $v\in V(B)$. Consider the bipartite graph $B(4)$ obtained from the disjoint union of $B$ and $|B|$ copies of $C_4$ by identifying the vertex $v$ with a unique vertex of $C_4^v$ (see Figure~\ref{fig:c4-extend} for the illustration of the graph $C_6(4)$). The graph $B(4)$ is clearly a $\csbe$-graph in which $B$ is an induced subgraph. 

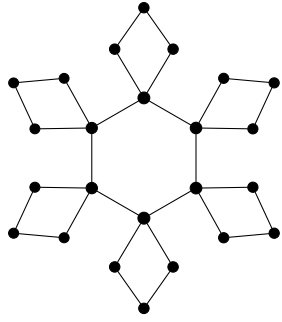
\begin{figure}[ht]
\centering  
\begin{tikzpicture}

    % Coordinates for the vertices of C6 (hexagonal shape)
    \node[circle, draw, scale=0.3, fill, inner sep = 1pt] (A) at (90:0.8) {A};
    \node[circle, draw, scale=0.3, fill, inner sep = 1pt] (B) at (30:0.8) {B};
    \node[circle, draw, scale=0.3, fill, inner sep = 1pt] (C) at (330:0.8) {C};
    \node[circle, draw, scale=0.3, fill, inner sep = 1pt] (D) at (270:0.8) {D};
    \node[circle, draw, scale=0.3, fill, inner sep = 1pt] (E) at (210:0.8) {E};
    \node[circle, draw, scale=0.3, fill, inner sep = 1pt] (F) at (150:0.8) {F};
    
    % Edges of C6
    \draw (A) -- (B);
    \draw (B) -- (C);
    \draw (C) -- (D);
    \draw (D) -- (E);
    \draw (E) -- (F);
    \draw (F) -- (A);

    \node[circle, draw, scale=0.2, fill, inner sep = 1pt] (A1) at (75:1.5) {A1};
    \node[circle, draw, scale=0.2, fill, inner sep = 1pt] (A2) at (105:1.5) {A2};
    \node[circle, draw, scale=0.2, fill, inner sep = 1pt] (A3) at (90:2) {A3};
    \draw (A) -- (A1) -- (A3) -- (A2) -- (A);
    
    % C4 at B
    \node[circle, draw, scale=0.2, fill, inner sep = 1pt] (B1) at (15:1.5) {B1};
    \node[circle, draw, scale=0.2, fill, inner sep = 1pt] (B2) at (45:1.5) {B2};
    \node[circle, draw, scale=0.2, fill, inner sep = 1pt] (B3) at (30:2) {B3};
    \draw (B) -- (B1) -- (B3) -- (B2) -- (B);
    
    % C4 at C
    \node[circle, draw, scale=0.2, fill, inner sep = 1pt] (C1) at (315:1.5) {C1};
    \node[circle, draw, scale=0.2, fill, inner sep = 1pt] (C2) at (345:1.5) {C2};
    \node[circle, draw, scale=0.2, fill, inner sep = 1pt] (C3) at (330:2) {C3};
    \draw (C) -- (C1) -- (C3) -- (C2) -- (C);
    
    % C4 at D
    \node[circle, draw, scale=0.2, fill, inner sep = 1pt] (D1) at (255:1.5) {D1};
    \node[circle, draw, scale=0.2, fill, inner sep = 1pt] (D2) at (285:1.5) {D2};
    \node[circle, draw, scale=0.2, fill, inner sep = 1pt] (D3) at (270:2) {D3};
    \draw (D) -- (D1) -- (D3) -- (D2) -- (D);
    
    % C4 at E
    \node[circle, draw, scale=0.2, fill, inner sep = 1pt] (E1) at (195:1.5) {E1};
    \node[circle, draw, scale=0.2, fill, inner sep = 1pt] (E2) at (225:1.5) {E2};
    \node[circle, draw, scale=0.2, fill, inner sep = 1pt] (E3) at (210:2) {E3};
    \draw (E) -- (E1) -- (E3) -- (E2) -- (E);
    
    % C4 at F
    \node[circle, draw, scale=0.2, fill, inner sep = 1pt] (F1) at (135:1.5) {F1};
    \node[circle, draw, scale=0.2, fill, inner sep = 1pt] (F2) at (165:1.5) {F2};
    \node[circle, draw, scale=0.2, fill, inner sep = 1pt] (F3) at (150:2) {F3};
    \draw (F) -- (F1) -- (F3) -- (F2) -- (F);

\end{tikzpicture}
\caption{The $\csbe$-graph $C_6(4)$.}
\label{fig:c4-extend}
\end{figure}

We further note that if $\E$ is a $\csbe$-sequence in a bipartite graph $H$, then it forms an induced matching of $H$ so that $|\E|\leq \im(H)$, while this inequality could be strict. For instance, consider the disjoint union of stars $K_{1,n}$ and $K_{1,n+1}$ for $n\geq 3$ over the vertex sets $\{x,a_i\colon 1\leq i\leq n\}$ and $\{y,b_j\colon 1\leq j\leq n+1\}$, where $x$ and $y$ are the respective center vertices. Let $H_n$ be the graph obtained from $K_{1,n}\cup K_{1,n+1}$ by the addition of extra edges $a_ib_i$ for every $1\leq i\leq n$. It follows that $\E:=\{yb_{n+1},xa_1\}$ is a $\csbe$-sequence in $H_n$, while $\im(H_n)=n$. 

\begin{lemma}\label{lem:dim}
If a bipartite graph $B$ admits a complete simple biclique elimination sequence  of length $k$, then $\Ind(S(B))\simeq \SA^{|B|-k-1}$.
\end{lemma}
\begin{proof}
Let $\E=\{e_1,e_2,\ldots,e_k\}$ be a complete simple biclique elimination sequence of $B$. For each $1\leq i\leq k$, we set $B_i:=B_{i-1}-N_{B_{i-1}}[e_i]$ and $p_i=|N_{B_{i-1}}[e_i]|$ as earlier, where $B_0:=B$. Following the proof Theorem~\ref{thm:main-2-a}, we may obtain 
$$\Ind(S(B)) \simeq \Sigma^{p_1-1} (\Ind(S(B-N_B[e_1]))= \Sigma^{p_1-1} (\Ind(S(B_1)).$$
Similarly, we have 
$$\Ind(S(B_1))\simeq \Sigma^{p_2-1} (\Ind(S(B_1-N_B[e_2]))= \Sigma^{p_2-1} (\Ind(S(B_2)),$$
which implies that $\Ind(S(B)) \simeq \Sigma^{p_1+p_2-2}(\Ind(S(B_2))$. Applying the same argument for each but the last step of the elimination, we get
$$\Ind(S(B)) \simeq \Sigma^{p_1+\ldots+p_{k-1}-(k-1)}(\Ind(S(B_{k-1})).$$
Finally, since $B_{k-1}$ is itself a biclique, we have $\Ind(S(B_{k-1}))\simeq \SA^{p_k-2}$ by Proposition~\ref{prop:subd-comp}. Therefore, we conclude
$$\Ind(S(B))\simeq \Sigma^{p_1+\ldots+p_{k-1}-k+1}(\SA^{p_k-2})\simeq \SA^{p_1+\ldots+p_k-k-1} \simeq \SA^{|B|-k-1}.$$
\end{proof}

\begin{proposition}\label{prop:sbe}
If $B$ admits a complete simple biclique elimination sequence $\E$, then every simple biclique elimination sequence in $B$ is complete of length $|\E|$. 
\end{proposition}
\begin{proof}
Let $B$ admit a complete simple biclique elimination sequence $\E=\{e_1,e_2,\ldots,e_k\}$. Then the complex $\Ind(S(B))$ is homotopy equivalent to a sphere by Lemma~\ref{lem:dim}. Now assume that $B$ has a simple biclique elimination sequence which is not complete. Then following the proof of Theorem~\ref{thm:main-2-a}, we conclude that $\Ind(S(B))$ must be contractible, which is a contradiction. Therefore, if $B$ has a simple biclique elimination sequence, all such sequences of $B$ must be complete. 

We now claim that every complete simple biclique elimination sequence of $B$ has the same length $|\E|$. Pick any two such distinct sequences $\E=\{e_1,e_2,\ldots,e_k\}$ and $\F=\{f_1,...,f_r\}$ of $B$. By Lemma~\ref{lem:dim}, we both have $\Ind(S(B))\simeq \SA^{|B|-k-1}$ and $\Ind(S(B))\simeq \SA^{|B|-r-1}$, which implies that $k=r$.
\end{proof}

\begin{corollary}\label{cor:csbe-2}
Let $B$ be a bipartite graph with a simple biclique elimination sequence. Then, the complex $\Ind(S(B))$ is homotopy equivalent to a sphere  if and only if $B$ admits a complete simple biclique elimination sequence.
\end{corollary}
\begin{proof}
Suppose that $B$ has a complete simple biclique elimination sequence. By Lemma \ref{lem:dim}, the complex $\Ind(S(B))$ is homotopy equivalent to a sphere.
For the converse, assume that $B$ has a simple biclique elimination sequence which is not complete. Then by the proof of Theorem~\ref{thm:main-2-a}, the complex $\Ind(S(B))$ must be contractible, a contradiction. 
\end{proof}

In the view of Corollary~\ref{cor:csbe-2}, in order to complete the proof of Theorem~\ref{thm:main-2}, it is sufficient to show that we can determine whether a bipartite graph admits a complete simple biclique elimination sequence in polynomial time.\medskip 

\textsc{csbe}

{\em Instance}: A bipartite graph $B$;

{\em Question}: Is the graph $B$ a $\csbe$-graph?

\begin{lemma}\label{lem:algo}
The \textsc{csbe} problem is polynomially solvable.
\end{lemma}
\begin{proof}
We consider the following algorithm. Let $B$ be an arbitrary bipartite graph.

\begin{enumerate}\itemsep -.5pt
\item let $\ell=\{\}$ be an empty list;
\item find a bisimplicial edge $e=uv$ of the graph $B$ such that $u$ or $v$ is a simple vertex and the graph $B-N_B[e]$ does not contain isolated vertices;
      \item if there is no such an edge $e$, then stop; if $B$ has no vertices then  return $\ell$, otherwise print ``the initial graph has no any $\csbe$-sequence'';
  \item add the edge $e$ to the list $\ell$;
  \item delete from the graph $B$ the vertices of the set $N_B[e]$;
    \item go to step $2$.
    \end{enumerate}

It is not hard to see that if the graph $B$ has no any $\csbe$-sequence, then the algorithm, correctly, returns the message ``the graph has no any $\csbe$-sequence''. It remains only to show that for each $\csbe$-graph $B$, the algorithm finds a $\csbe$-sequence of the graph $B$. We use induction on the size $t=|\E(B)|$ of a $\csbe$-sequence of $B$. If $t=1$, then the graph $B$ is a complete bipartite graph. The list $\ell$ ($|\ell|=1$) returned by the algorithm is a $\csbe$-sequence of the graph $B$. Let $t \geq 2$. We consider a $\csbe$-sequence $\E(B)=\{e_1,e_2,\ldots,e_t\}$ with $e_i=u_iv_i$ of the graph $B$. We assume that the vertices $u_1,u_2,\ldots,u_t$ are at the same part of the bipartite graph $B$. In particular, this implies that $\{e_1, e_2, \ldots,e_t\}$ is an induced matching. Once again, we use the following notations: $B_0 = B$ and $B_i=B_{i-1}-N_{B_{i-1}}[e_i]$ for each $i\in [t]$. Notice that induced subgraphs $B_1,\ldots,B_t$ form a biclique vertex partition of the graph $B$. We say that this biclique vertex partition is {\em induced} by $\E(B)$.\medskip

{\it Claim} $1$: Each bisimplicial edge of the graph $B$ is an edge of a biclique $B_i$ for some $i\in \{0,1,2,\ldots,t-1\}$.
\medskip

{\it Proof of Claim} $1$:
Suppose for the contrary that an edge $e=uv$ is a bisimplicial edge of the graph $B$ such that  the vertex $u$ is a vertex of a biclique $B_{i-1}$ and the vertex $v$ is a vertex of a biclique $B_{j-1}$ for some $i, j \in [t]$ with $i \neq j$. There are only two possible cases, both of which are illustrated in Figure~\ref{fig: 2cases}.
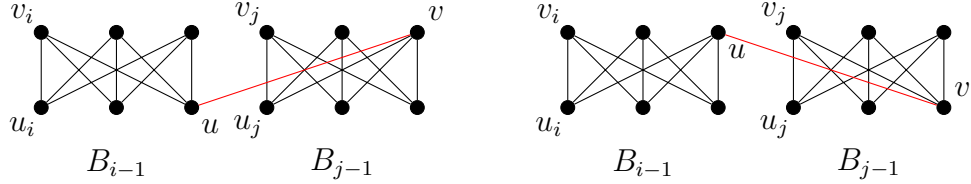
\begin{figure}[ht]
\begin{center}
\begin{tikzpicture}
          \node[draw, fill, circle, inner sep = 0.65mm] (a1) at (0, 0) {};
          \node[draw, fill, circle, inner sep = 0.65mm] (b1) at (1, 0) {};
          \node[draw, fill, circle, inner sep = 0.65mm] (c1) at (2, 0) {};

          \node[draw, fill, circle, inner sep = 0.65mm] (a2) at (0, 1) {};
          \node[draw, fill, circle, inner sep = 0.65mm] (b2) at (1, 1) {};
          \node[draw, fill, circle, inner sep = 0.65mm] (c2) at (2, 1) {};

          \node at (-0.25, -0.25) {$u_i$};
          \node at (-0.25, 1+0.25) {$v_i$};
          \node at (1, -0.75) {$B_{i-1}$};
          
          \draw (a1) -- (a2)
          (a1) -- (b2)
          (a1) -- (c2)
          (b1) -- (a2)
          (b1) -- (b2)
          (b1) -- (c2)
          (c1) -- (a2)
          (c1) -- (b2)
          (c1) -- (c2);

          %%%%%%%%%%%%%%%%%%%%%%%%%%%%%%%%

          \node[draw, fill, circle, inner sep = 0.65mm] (a1s) at (0 + 3, 0) {};
          \node[draw, fill, circle, inner sep = 0.65mm] (b1s) at (1 + 3, 0) {};
          \node[draw, fill, circle, inner sep = 0.65mm] (c1s) at (2 + 3, 0) {};

          \node[draw, fill, circle, inner sep = 0.65mm] (a2s) at (0 + 3, 1) {};
          \node[draw, fill, circle, inner sep = 0.65mm] (b2s) at (1 + 3, 1) {};
          \node[draw, fill, circle, inner sep = 0.65mm] (c2s) at (2 + 3, 1) {};

          \node at (-0.25 + 3, -0.25) {$u_j$};
          \node at (-0.25 + 3, 1+0.25) {$v_j$};
          \node at (1 + 3, -0.75) {$B_{j-1}$};
          
          \draw (a1s) -- (a2s)
          (a1s) -- (b2s)
          (a1s) -- (c2s)
          (b1s) -- (a2s)
          (b1s) -- (b2s)
          (b1s) -- (c2s)
          (c1s) -- (a2s)
          (c1s) -- (b2s)
          (c1s) -- (c2s);

          \node at (2+ .25, 0 - .25) {$u$};
          \node at (2 + + 3 .25, 1 + .25) {$v$};
          
          \draw[red] (c1) -- (c2s);

          %%%%%%%%%%%%%%%%%%%%%%%%%%%%%%%%%
          %%%%%%%%%%%%%%%%%%%%%%%%%%%%%%%%%

          \node[draw, fill, circle, inner sep = 0.65mm] (a1) at (0 + 7, 0) {};
          \node[draw, fill, circle, inner sep = 0.65mm] (b1) at (1 + 7, 0) {};
          \node[draw, fill, circle, inner sep = 0.65mm] (c1) at (2 + 7, 0) {};

          \node[draw, fill, circle, inner sep = 0.65mm] (a2) at (0 + 7, 1) {};
          \node[draw, fill, circle, inner sep = 0.65mm] (b2) at (1 + 7, 1) {};
          \node[draw, fill, circle, inner sep = 0.65mm] (c2) at (2 + 7, 1) {};

          \node at (-0.25 + 7, -0.25) {$u_i$};
          \node at (-0.25 + 7, 1+0.25) {$v_i$};
          \node at (1 + 7, -0.75) {$B_{i-1}$};
          
          \draw (a1) -- (a2)
          (a1) -- (b2)
          (a1) -- (c2)
          (b1) -- (a2)
          (b1) -- (b2)
          (b1) -- (c2)
          (c1) -- (a2)
          (c1) -- (b2)
          (c1) -- (c2);

          %%%%%%%%%%%%%%%%%%%%%%%%%%%%%%%%

          \node[draw, fill, circle, inner sep = 0.65mm] (a1s) at (0 + 3 + 7, 0) {};
          \node[draw, fill, circle, inner sep = 0.65mm] (b1s) at (1 + 3 + 7, 0) {};
          \node[draw, fill, circle, inner sep = 0.65mm] (c1s) at (2 + 3 + 7, 0) {};

          \node[draw, fill, circle, inner sep = 0.65mm] (a2s) at (0 + 3 + 7, 1) {};
          \node[draw, fill, circle, inner sep = 0.65mm] (b2s) at (1 + 3 + 7, 1) {};
          \node[draw, fill, circle, inner sep = 0.65mm] (c2s) at (2 + 3 + 7, 1) {};

          \node at (-0.25 + 3 + 7, -0.25) {$u_j$};
          \node at (-0.25 + 3 + 7, 1+0.25) {$v_j$};
          \node at (1 + 3 + 7, -0.75) {$B_{j-1}$};
          
          \draw (a1s) -- (a2s)
          (a1s) -- (b2s)
          (a1s) -- (c2s)
          (b1s) -- (a2s)
          (b1s) -- (b2s)
          (b1s) -- (c2s)
          (c1s) -- (a2s)
          (c1s) -- (b2s)
          (c1s) -- (c2s);

          \node at (2 + 7 + .25, 1 - .25) {$u$};
          \node at (2 + 7 + 3 .25, 0 + .25) {$v$};
          
          \draw[red] (c2) -- (c1s);

\end{tikzpicture}
\end{center}
\caption{Two possible cases of Claim $1$.}
\label{fig: 2cases}
\end{figure}

As the edge $e=uv$ is bisimplicial, in the first case the vertices $v_i$ and $u_j$ are adjacent, and in the second case the vertices $u_i$ and $v_j$ are adjacent. However, this leads to a contradiction, since $\{e_1,e_2,\ldots,e_k\}$ is an induced matching. This  completes the proof of Claim $1$.

Let $e=uv$ be the bisimplicial edge of the graph $B$ found by the algorithm at the beginning. By the Claim $1$, the edge $e$ is an edge of the biclique $B_{i-1}$ for some $i\in [t]$. It follows that $u_iv, v_iu\in E(B)$. As the edge $e$ is bisimplicial and $\{e_1,e_2,\ldots,e_t\}$ is an induced matching, the vertex $u$ can not be adjacent to vertices $v_1,v_2,\ldots,v_{i-1},v_{i+1},\ldots,v_t$, and similarly, the vertex $v$ can not be adjacent to vertices $u_1, u_2,\ldots,u_{i-1}, u_{i+1},\ldots,u_t$. Consequently, the graph $B_1=B-N_{B}[e]$ is a $\csbe$-graph. The sequence $e_1,e_2,\ldots,e_{i-1}, e_{i+1}, \ldots, e_t$ is a $\csbe$-sequence of the graph $B_1$. By the induction hypothesis, the algorithm finds a $\csbe$-sequence $\ell$ of the graph $B_1$. As a result, the algorithm returns the $\csbe$-sequence $\{e\} \cup \ell$ of the graph $B$.

Finally, we can check in $O(n^2)$ time whether an arbitrary edge of an $n$-vertex bipartite graph $B$ is bisimplicial and contains a simple vertex~\citep{bombod}. Since a complete simple biclique elimination sequence $\E$ contains at most $\frac{n}{2}$ edges, we can construct such a sequence in $O(n^4)$ time if $B$ admits it. 
\end{proof}

Let $B$ be a chordal bipartite graph. If $B$ is used as an input in the algorithm provided in Lemma~\ref{lem:algo}, the algorithm determines in polynomial time whether $B$ is a $\csbe$-graph. If $B$ is not
a $\csbe$-graph, then $\Ind(S(B))$ is clearly contractible by Theorem~\ref{thm:main-2-a} and Corollary~\ref{cor:csbe-2}. Otherwise, the algorithm returns (the size of) a complete simple biclique elimination sequence of $B$ that also determines the dimension of the sphere to which $\Ind(S(B))$ is homotopy equivalent. This concludes the proof of Theorem~\ref{thm:main-2}.\medskip

When a bipartite graph $B$ is a $\csbe$-graph, we denote by $|\E(B)|$,
the length of the complete simple biclique elimination sequence in $B$.

\begin{remark}
The difference $|\mathcal{E}(B)|-\bp(B)$ could be arbitrarily large  even for chordal bipartite $\csbe$-graphs. Consider the bipartite graph $B_p$
given by
\begin{displaymath}
V(B_p)=\{a,b,c_1,\ldots,c_p,d_1,\ldots,d_p\}\quad \textnormal{and}\quad E(B_p):=\{ab, bc_i,c_id_i,d_ia\colon 1\leq i\leq p\}
\end{displaymath}
for $p>2$ (the graph $B_4$ is depicted in Figure~\ref{fig:b_4}). We note that $B_p$ is chordal bipartite and $|\mathcal{E}(B_p)|-\bp(B_p) = p-2$.
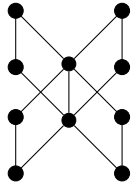
\begin{figure}[ht]
  \begin{center}
      \begin{tikzpicture}
       
        \node[ regular polygon,
      regular polygon sides = 4,
      minimum size = 2cm,
      inner sep=0mm,
      outer sep=0mm,
      rotate=0,
      thick,
      color = white,
      draw] (A) at (0, 0) {};

       \node[ regular polygon,
      regular polygon sides = 4,
      minimum size = 2cm,
      inner sep=0mm,
      outer sep=0mm,
      rotate=0,
      thick,
      color = white,
      draw] (B) at (0, -0.75) {};

       \node[draw, fill, circle, inner sep = 0.65mm] (s) at (0, 0) {};
        \node[draw, fill, circle, inner sep = 0.65mm] (t) at (0, -0.75) {};

      \foreach \i in {1,...,4}{
        \node[draw, circle, fill, inner sep = 2pt] at (A.corner \i) {};
      }

      \foreach \i in {1,...,4}{
        \node[draw, circle, fill, inner sep = 2pt] at (B.corner \i) {};
      }

      \draw (t) -- (B.corner 1)
      (t) -- (B.corner 2)
      (t) -- (B.corner 3)
      (t) -- (B.corner 4);

      \draw (s) -- (A.corner 1)
      (s) -- (A.corner 2)
      (s) -- (A.corner 3)
      (s) -- (A.corner 4);

      \draw (s) -- (t)
      (A.corner 1) -- (B.corner 1)
      (A.corner 2) -- (B.corner 2)
      (A.corner 3) -- (B.corner 3)
      (A.corner 4) -- (B.corner 4);
        \end{tikzpicture}
      \end{center}
      \caption{The chordal bipartite graph $B_4$.}
\label{fig:b_4}
\end{figure}
\end{remark}

We end this section by describing a subclass of bipartite graphs in which the biclique vertex partition number and the length of a complete simple biclique elimination sequence (when exists) coincide. We achieve that by means of introducing a forbidden family of graphs in which these two parameters differ whenever they admit such a sequence.

Let $\Gamma$ be a family of bipartite graphs such that if $B\in \Gamma$, then the followings hold:
\begin{itemize}
\item[$(i)$] $B$ is connected with $\delta(B)\geq 2$,\\
\item[$(ii)$] there exists an induced matching $M$ in $B$ with $|M|>\frac{|B|}{3}$,\\
\item[$(iii)$] for each vertex $v\in V(B)\setminus V(M)$, there exist at least two edges $xy,wz\in M$ such that $N_B(v)\cap \{x,y\}\neq \emptyset$ and $N_B(v)\cap \{w,z\}\neq \emptyset$.
\end{itemize}
The condition $(iii)$ amounts to saying that every vertex of $B$ is either incident to an edge in $M$ or it is adjacent to one of the end vertices of at least two edges in $M$. In Figure~\ref{fig:Gamma}, we depict three examples of graphs from the class $\Gamma$. Note also that the family $\Gamma$ may contain bipartite graphs without any simple biclique elimination. Throughout, we write $(B,M)\in \Gamma$ to indicate that the bipartite graph $B$ is contained in $\Gamma$ and $M$ is an induced matching of $B$ satisfying the required conditions. 

\begin{figure}[ht]
\begin{center}
  \begin{tikzpicture}[scale=.9]
    \def \d  {-4.5};
    \node[circle, draw, fill, inner sep = 1.75pt] at (0 + \d, 0) {};
    \node[circle, draw, fill, inner sep = 1.75pt] at (1 + \d, 0) {};
    \node[circle, draw, fill, inner sep = 1.75pt] at (2 + \d, 0) {};
    \node[circle, draw, fill, inner sep = 1.75pt] at (0 + \d, 1) {};
    \node[circle, draw, fill, inner sep = 1.75pt] at (1 + \d, 1) {};
    \node[circle, draw, fill, inner sep = 1.75pt] at (2 + \d, 1) {};

    \node[circle, draw, fill, inner sep = 1.75pt] at (1 + \d, 2) {};
    \node[circle, draw, fill, inner sep = 1.75pt] at (1 + \d, -1) {};
    
    \draw (0 + \d, 0) -- (0 + \d, 1)
    (1 + \d, 0) -- (1 + \d, 1)
    (2 + \d, 0) -- (2 + \d, 1)
    (1 + \d, 2) -- (0 + \d, 1)
    (1 + \d, 2) -- (1 + \d, 1)
    (1 + \d, 2) -- (2 + \d, 1)
    (1 + \d, -1) -- (0 + \d, 0)
    (1 + \d, -1) -- (1 + \d, 0)
    (1 + \d, -1) -- (2 + \d, 0);

%    \node at (-4.75, 0.5) {$e_1$};
%    \node at (-4.75 + 1, 0.5) {$e_2$};
%    \node at (-4.75 + 2, 0.5) {$e_3$};
    %%%%%%%%%%%%%%%%%%%
    
    \node[circle, draw, fill, inner sep = 1.75pt] at (0, 0) {};
    \node[circle, draw, fill, inner sep = 1.75pt] at (1, 0) {};
    \node[circle, draw, fill, inner sep = 1.75pt] at (2, 0) {};
    \node[circle, draw, fill, inner sep = 1.75pt] at (0, 1) {};
    \node[circle, draw, fill, inner sep = 1.75pt] at (1, 1) {};
    \node[circle, draw, fill, inner sep = 1.75pt] at (2, 1) {};

    \node[circle, draw, fill, inner sep = 1.75pt] at (1, 2) {};
    \node[circle, draw, fill, inner sep = 1.75pt] at (1, -1) {};
    
    \draw (0, 0) -- (0, 1)
    (1, 0) -- (1, 1)
    (2, 0) -- (2, 1)
    (1, 2) -- (0, 1)
    (1, 2) -- (1, 1)
    (1, 2) -- (2, 1)
    (1, -1) -- (0, 0)
    (1, -1) -- (1, 0)
    (1, -1) -- (2, 0);

    \draw (1, -1) .. controls (3.5, -1) and (3.5, 2) .. (1, 2);

%    \node at (-4.75 + 4.5, 0.5) {$e_1$};
%    \node at (-4.75 + 1 + 4.5, 0.5) {$e_2$};
%    \node at (-4.75 + 2 + 4.5, 0.5) {$e_3$};
    
    %%%%%%%%%%%%%%%%%%%%%%%%%%%

    \node[circle, draw, fill, inner sep = 1.75pt] (a1) at (6, 0) {};
    \node[circle, draw, fill, inner sep = 1.75pt] (a2) at (6, 1) {};
    \node[circle, draw, fill, inner sep = 1.75pt] (b1) at (7, 0) {};
    \node[circle, draw, fill, inner sep = 1.75pt] (b2) at (7, 1) {};
    \node[circle, draw, fill, inner sep = 1.75pt] (c1) at (8, 0) {};
    \node[circle, draw, fill, inner sep = 1.75pt] (c2) at (8, 1) {};
    \node[circle, draw, fill, inner sep = 1.75pt] (d1) at (9, 0) {};
    \node[circle, draw, fill, inner sep = 1.75pt] (d2) at (9, 1) {};

    \node[circle, draw, fill, inner sep = 1.75pt] (top) at (7.5, 2) {};
    \node[circle, draw, fill, inner sep = 1.75pt] (bottom1) at (6.5, -1) {};
    \node[circle, draw, fill, inner sep = 1.75pt] (bottom2) at (8.5, -1) {};

    \draw (top) -- (a2)
    (top) -- (b2)
    (top) -- (c2)
    (top) -- (d2)
    (a1) -- (a2)
    (b1) -- (b2)
    (c1) -- (c2)
    (d1) -- (d2)
    (bottom1) -- (a1)
    (bottom1) -- (b1)
    (bottom2) -- (c1)
    (bottom2) -- (d1);
    
    \draw (bottom1) .. controls (4, -1) and (4, 2) .. (top);
    \draw (bottom2) .. controls (11, -1) and (11, 2) .. (top);

%   \node at (-4 + 0.5, -2) {$B_1$};
%   \node at (-4 + 0.5 + 4.5, -2) {$B_2$};
%   \node at (-4 + 0.5 + 11, -2) {$B_3$};

%    \node at (-4.75 + 10.5, 0.5) {$e_1$};
%    \node at (-4.75 + 1 + 10.5, 0.5) {$e_2$};
%    \node at (-4.75 + 2 + 10.5, 0.5) {$e_3$};
%    \node at (-4.75 + 3 + 10.5, 0.5) {$e_4$};
\end{tikzpicture}
\end{center}
\caption{Examples of graphs from the family $\Gamma$.}
\label{fig:Gamma}  
\end{figure}
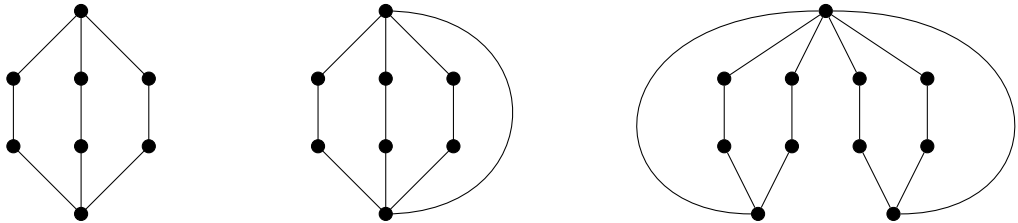

\begin{proposition}
The class $\Gamma$ contains an infinite family of graphs such that none of the proper induced subgraph of a graph from the family is a member of $\Gamma$. 
\end{proposition}
\begin{proof}
Let $k\geq 3$ be an integer and let $H_k$ be the graph obtained from the cycle $C_{6k}$ by identifying its two vertices $u$ and $v$ at distance exactly six. Observe that $H_k\in \Gamma$.

We claim that there exists no induced subgraph of $H_k$ contained in $\Gamma$. Indeed, deleting any set of vertices from $H_k$ leads to either a cycle or a graph with minimum degree one, neither of which belongs to $\Gamma$.
\end{proof}

\begin{fact}\label{prop:contra}
Let $(B,M)\in \Gamma$ be given. If $U=V\setminus V(M)$ is the set of vertices of $B$ that are not incident to any edges from the matching $M$, then $|M|>|U|\geq 2$.
\end{fact}
\begin{proof}
We have that $|M|>\frac{|B|}{3}=\frac{2|M|+|U|}{3}$. It implies that $|M|>|U|$. As the graph $B$ is bipartite, $\delta(B) \geq 2$ and the matching $M$ is induced, we conclude that $|U|\geq 2$.
\end{proof}

\begin{proposition}\label{prop:ebp}
If a graph $B\in \Gamma$ has a $\csbe$-sequence $\E$, then $|\E|>\bp(G)$.
\end{proposition}
\begin{proof}
Suppose that $(B,M)\in \Gamma$ admits a $\csbe$-sequence $\E$. We first verify that $\E\subseteq M$.

If we set $U:=V(B)\setminus V(M)$, then $E(B)$ can be partitioned into three sets $E_1=M$, $E_2=\{uw \in E(B)\colon u\in U\;\textnormal{and}\;w \in V(M)\}$ and $E_3=\{uu'\in E(B)\colon u,u'\in U\}$. 

Consider an edge $e=uw\in E_2$, where $u \in U$ and $w \in V(M)$. If the vertex $u$ is adjacent to end vertices, say $v$ and $w$, of at least two edges $ww', vv'\in M$, then the edge $e$ can not be bisimplicial (see Figure~\ref{fig:ebp}-$(a)$).
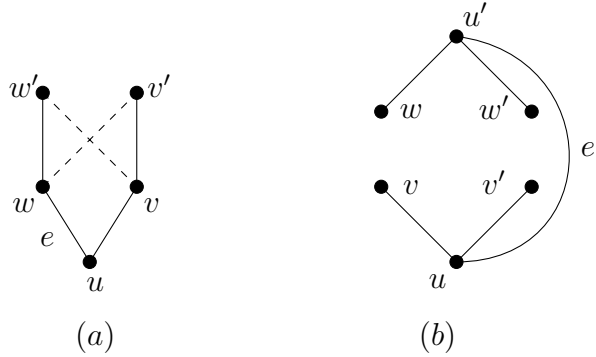
\begin{figure}[ht]
\begin{center}
\begin{tikzpicture}
              \node[circle, draw, fill, inner sep = 1.75pt] at (0, 0) {};
              \node[circle, draw, fill, inner sep = 1.75pt] at (1.25, 0) {};
              \node[circle, draw, fill, inner sep = 1.75pt] at (0, 1.25) {};
              \node[circle, draw, fill, inner sep = 1.75pt] at (1.25, 1.25) {};

              \node[circle, draw, fill, inner sep = 1.75pt] at (0.625, -1) {};

              \draw (0, 0) -- (0, 1.25)
              (1.25, 0) -- (1.25, 1.25)
              (0.625, -1) -- (0, 0)
              (0.625, -1) -- (1.25, 0);

              \draw[dashed] (0, 0) -- (1.25, 1.25)
              (0, 1.25) -- (1.25, 0);

              \node at (0.06, -0.7) {$e$};
              \node at (0.7, -1.3) {$u$};
              \node at (-0.25, -0.25) {$w$};
              \node at (-0.25, -0.25 + 1.6) {$w'$};
              \node at (-0.25 + 1.7, -0.25) {$v$};
              \node at (-0.25 + 1.8, -0.25 + 1.6) {$v'$};
 \node at (0.7, -2) {$(a)$};

              \def \d  {4.5};
    \node[circle, draw, fill, inner sep = 1.75pt] at (0 + \d, 0) {};
    %\node[circle, draw, fill, inner sep = 1.75pt] at (1 + \d, 0) {};
    \node[circle, draw, fill, inner sep = 1.75pt] at (2 + \d, 0) {};
    \node[circle, draw, fill, inner sep = 1.75pt] at (0 + \d, 1) {};
    %\node[circle, draw, fill, inner sep = 1.75pt] at (1 + \d, 1) {};
    \node[circle, draw, fill, inner sep = 1.75pt] at (2 + \d, 1) {};

    \node[circle, draw, fill, inner sep = 1.75pt] at (1 + \d, 2) {};
    \node[circle, draw, fill, inner sep = 1.75pt] at (1 + \d, -1) {};
    
    \draw %(0 + \d, 0) -- (0 + \d, 1)
    %(1 + \d, 0) -- (1 + \d, 1)
    %(2 + \d, 0) -- (2 + \d, 1)
    (1 + \d, 2) -- (0 + \d, 1)
    %(1 + \d, 2) -- (1 + \d, 1)
    (1 + \d, 2) -- (2 + \d, 1)
    (1 + \d, -1) -- (0 + \d, 0)
    %(1 + \d, -1) -- (1 + \d, 0)
    (1 + \d, -1) -- (2 + \d, 0);
    
    \draw (1 + \d, 2) .. controls (3 + \d, 1.75) and (3 + \d, -1) .. (1 + \d, -1);
    \node at (1 + \d + 0.25, 2 + 0.3) {$u'$};
    \node at  (1 + \d - 0.25, -1 - 0.25) {$u$};
    \node at (0 + \d + 0.4, 1) {$w$};
    \node at (0 + \d + 0.4, 0) {$v$};
    \node at (0 + \d + 0.4 + 1.1, 1.06) {$w'$};
    \node at (0 + \d + 0.4 + 1.1, 0.06) {$v'$};
    \node at (3 + \d -0.25, 1 - 0.5) {$e$};
    \node at  (1 + \d - 0.25, -1 - 1) {$(b)$};
\end{tikzpicture}
\end{center}
\caption{Illustrations for the proof of Proposition~\ref{prop:ebp}.}
\label{fig:ebp}  
\end{figure} 
         
Next we consider an edge $e=uu'\in E_3$, where $u,u'\in U$. If each end vertex $u$ and $u'$ is adjacent to end vertices of at least two edges from the matching $M$, say $u$ is adjacent to vertices $v,v'\in V(M)$ and the vertex $u'$ is adjacent to vertices $w,w'\in V(M)$, then once again the edge $e$ can not be bisimplicial (see Figure~\ref{fig:ebp}-$(b)$). 

As a result we conclude that an edge in $E_2\cup E_3$ can be bisimplicial only in the case when there is a vertex in $U$ adjacent to one of the end vertices of at most one edge from the matching $M$.
Since $(B,M)\in \Gamma$, each vertex of $U$ must be adjacent to one of the end vertices of at least two edges in $M$. Therefore, any bisimplicial edge of $B$ is contained in $E_1=M$. In other words, we have that $\E\subseteq M$ as we claimed. However, since $\E$ is a $\csbe$-sequence of $B$, this implies that $\E=M$.

Finally, the equality $\E=M$ forces that the set $U$ is a dominating set in $B$. Thus, we conclude that
\begin{displaymath}
\bp(B)\leq \gamma(B)\leq |U|<|M|=|\E|
\end{displaymath}
by Fact~\ref{prop:contra}.
\end{proof}

\begin{lemma}\label{lem:gamma}
Let $B$ be a bipartite $\csbe$-graph with a $\csbe$-sequence $\E$. If $B$ is $\Gamma$-free, then $|\E|=\bp(B)$.
\end{lemma}
\begin{proof}
It suffices to verify that if $|\E|>\bp(B)$, then $B$ contains at least one graph from the class $\Gamma$ as an induced subgraph.

Let $\SE$ be a biclique vertex partition of $B$ with $|\SE|=\bp(B)$. Recall that the set $\E$ forms an induced matching in $B$. Assume that
\begin{equation}\label{eq:assume}
|\SE|=\bp(B)<|\E|.
\end{equation}
Let $\SE'$ be the set of bicliques in $\SE$ containing at least one of the end vertices of edges in $\E$. The inequality (\ref{eq:assume}) implies that
\begin{equation}\label{eq:ineq:1}
|\SE'|<|\E|.
\end{equation}

We transform the sets $\SE'$ and $\E$ iteratively by the following rule. If the set $\SE'$ has the biclique containing end vertices of exactly one edge from $\E$, say the end vertices of the edge $e \in \E$, then we delete that biclique from $\SE'$ and the edge $e$ from $\E$. We denote resulting sets by $\SE''$ and $\E''$. It is not hard to see that the inequality (\ref{eq:ineq:1}) implies the inequality $|\SE''|<|\E''|$. Each end vertex of any edge of $\E''$ is contained in exactly one of bicliques of $\SE''$ and each biclique in $\SE''$ contains end vertices of at least two edges of $\E''$.

Next, we construct a vertex set $U$ based on the sets $\SE''$ and $\E''$. Let $U_{\E}$ be the set of end vertices of edges in $\E''$. Take one vertex from each biclique in $\SE''$, adjacent to at least two vertices of $U_{\E}$, and denote by $U_{\SE}$, the set of chosen vertices. We have $|U_{\E}|=2|\E''|$ and $|U_{\SE}|=|\SE''|$ so that
\begin{equation}\label{eq:base}
2|U_{\SE}|<|U_{\E}|
\end{equation}
holds.

We set $U:=U_{\E} \cup U_{\SE}$. Observe that the subgraph of $B$ induced by $U$ may not be connected. In such a case, we consider one of the connected components of its induced subgraph by the vertex set $U'_{\E}\cup U'_{\SE}$, where $U'_{\E}\subseteq U_{\E}$ and $U'_{\SE}\subseteq U_{\SE}$ such that $2|U'_{\SE}|<|U'_{\E}|$ holds. In order to simplify the notation, we may abbreviate $U'_{\E}$, $U'_{\SE}$ and $U'_{\E}\cup U'_{\SE}$ to $U_{\E}$, $U_{\SE}$ and $U$, respectively.

We claim that the subgraph $B[U]$ of $B$ induced by $U$ is a member of the family $\Gamma$. The graph $B[U]$ is connected with the minimum degree at least two. Furthermore, the set $U_{\E}$ induces a $1$-regular subgraph and the set of edges in that subgraph forms an induced matching $M_U$ of size $|M_U|=\frac{|U_{\E}|}{2}>\frac{|U|}{3}$, since $|U|=|U_{\E}|+|U_{\SE}|<\frac{3}{2}|U_{\E}|$ and $\frac{|U_{\E}|}{2}>\frac{|U|}{3}$ by the inequality (\ref{eq:base}). Finally, each vertex in $U_{\SE}$ is adjacent to one of the end vertices of at least two edges of the matching $M_U$ by the construction.
\end{proof}

We denote by $S_{2,2,2}$, the \emph{long claw} obtained from the claw by subdividing every edge once, i.e., $S_{2,2,2}:=S(K_{1,3})$ (see Figure~\ref{fig: S_222}).
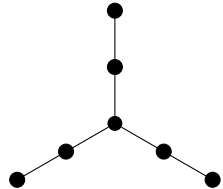
\begin{figure}[ht]
\begin{center}
\begin{tikzpicture}
        \node[draw, fill, circle, inner sep = 0.65mm] (s) at (0, 0) {};

        \node[ regular polygon,
      regular polygon sides = 3,
      minimum size = 1.5cm,
      inner sep=0mm,
      outer sep=0mm,
      rotate=0,
      thick,
      color = white,
      draw] (A) at (0, 0) {};

       \node[ regular polygon,
      regular polygon sides = 3,
      minimum size = 3cm,
      inner sep=0mm,
      outer sep=0mm,
      rotate=0,
      thick,
      color = white,
      draw] (B) at (0, 0) {};

      \foreach \i in {1,...,3}{
        \node[draw, circle, fill, inner sep = 2pt] at (A.corner \i) {};
      }

      \foreach \i in {1,...,3}{
        \node[draw, circle, fill, inner sep = 2pt] at (B.corner \i) {};
      }

      \draw (s) -- (B.corner 1)
      (s) -- (B.corner 2)
      (s) -- (B.corner 3);

\end{tikzpicture}
\end{center}
\caption{The long claw $S_{2,2,2}$.}
\label{fig: S_222}
\end{figure}      

\begin{corollary}\label{lem:s222}
If $\E$ is a $\csbe$-sequence of a $S_{2,2,2}$-free bipartite graph $B$, then $|\E|=\bp(B)$.
\end{corollary}
\begin{proof}
By Lemma~\ref{lem:gamma}, it is sufficient to show that each graph from the class $\Gamma$ contains the long claw $S_{2,2,2}$ as an induced subgraph.

So, let $(B,M)\in \Gamma$ be given. We claim that there exits a vertex of $B$ that is adjacent to one of the end vertices of at least three edges from the matching $M$. That vertex together with the end vertices of three edges induce a subgraph of $B$ isomorphic to the long claw $S_{2,2,2}$.

Let $U$ be a set of vertices of $B$ that are not incident to edges from the matching $M$. Suppose to the contrary that each vertex of $U$ is adjacent to one of the end vertices of at most two edges of the matching $M$. Since $\delta(B)\geq 2$, each end vertex of every edge from $M$ is adjacent to at least one of vertices in $U$. The total number of such end vertices equals to $2|M|$. On the other hand, each vertex of $U$ is adjacent to at most two of such vertices. Therefore, $|U|\geq |M|$, which contradicts to Fact~\ref{prop:contra}.
\end{proof}

%%%%%%%%%%%%%%%%%%%%%%%%%%%%%%%%%%%%%%%%%%%%%%%%%%

%%%%%%%%%%%%%%%%%%%%%%%%%%%%%%%%%%%%%%%%%%%%%%%%%%%%%%%%%
\section{Further Comments}\label{sect:concl}
We recall that when $X$ is a simplicial complex, its \emph{Alexander dual} is the complex defined by
\begin{displaymath}
X^\vee:=\{A\subseteq V\colon V\setminus A\notin X\}.
\end{displaymath}
It was proved by Csorba~\citep[Theorem $6$]{csorba} that the complex $\Ind(S(G))$ is homotopy equivalent to the suspension of $\Ind(G)^\vee$ for every graph $G$. Therefore, Theorem~\ref{thm:main-2-a} combined with Corollary~\ref{cor:csbe-2} and Lemma~\ref{lem:algo} also determines the homology of the complex $\Ind(B)^\vee$ for every bipartite graph $B$ with a simple biclique elimination sequence in polynomial time.

There is an another algebraic invariant that can be associated to graphs. Recall that if $X$ is a simplicial complex on a set $V$, the \emph{projective dimension} $\projdim(X)$ of $X$ is defined to the least integer $i$ such that
\begin{equation*}
\widetilde{H}_{|S|-i-j-1}(X[S])=0
\end{equation*}
for all $j>0$ and $S\subseteq V$. However, we note that the parameters $\projdim(G):=\projdim(\Ind(G))$ and $|G|-\bp(G)$ are incomparable in general. For instance, we have that $|C_{10}|-\bp(C_{10})=6<7=\projdim(C_{10})=\reg(S(C_{10}))$. On the other hand,
$\reg(S(DS_2))=|DS_2|-\bp(DS_2)=4>3=\projdim(DS_2)$, where $DS_2$ denotes the double-star graph, obtained from two disjoint copies of $K_{1,2}$ by adding an edge between their degree two vertices.\medskip

We naturally suspect that the class of chordal bipartite graphs is not the unique bipartite class for which the equality of Theorem~\ref{thm:main-1} holds. Recall that a bipartite graph $B=(X,Y;E)$ is called a \emph{tree-convex bipartite graph} (over $X$) if a tree $T$ defined on the vertex set $X$ exists, such that for every vertex $y\in Y$, the set $N_B(y)$ induces a subtree in $T$. Any specific choice for the tree $T$ gives rise to a subclass. For example, if $T$ is a path or a star, the corresponding subclasses are known as \emph{convex bipartite} and \emph{star-convex bipartite} graphs respectively.  

It is known that chordal bipartite graphs also form a subclass of tree-convex bipartite graphs~\citep[Theorem $5$]{jiang}. Therefore, we may ask the following question.

\begin{question}
For which trees, do the corresponding tree-convex bipartite graphs satisfy the equality of Theorem~\ref{thm:main-1}?
\end{question} 

We note that the family of convex bipartite graphs forms a subclass of chordal bipartite graphs~\citep[pp. $94$]{BLS-gclass}. In more detail, recall that a bipartite graph $B=(X,Y;E)$ is said to be \emph{convex} on $Y$ if the vertices of $Y$ can be ordered in such a way that the neighbours of any vertex $x\in X$ are consecutive. A bipartite graph $B$ is convex bipartite if it is convex on $X$ or $Y$. Furthermore, if $B$ is convex both on $X$ and $Y$, then it is called a \emph{biconvex bipartite graph}~\citep[pp. $93-94$]{BLS-gclass}. Notice that the graph depicted in Figure~\ref{fig:indom-bp} is a planar convex bipartite graph, which is not biconvex. 

\begin{question}
Does there exist a biconvex bipartite graph $B$ with $\gamma^i(B)<\bp(B)$?
\end{question}

%%%%%%%%%%%%%%%%%%%%%%%%%%%%%%%%%%%%%%%%%%%%%%%%%%%%%%%%%%%%%%%%%%%%%
\bibliographystyle{abbrv}  
\bibliography{biclique-parti}

\begin{thebibliography}{10}

\bibitem{aharoni}
R.~Aharoni, E.~Berger, and R.~Ziv.
\newblock A tree version of {K}{\"{o}}nig's theorem.
\newblock {\em Combinatorica}, 22(3):335--343, 2002.

\bibitem{BC-prime}
T.~B{\i}y{\i}ko{\u{g}}lu and Y.~Civan.
\newblock {C}astelnuovo-{M}umford regularity of graphs.
\newblock {\em Combinatorica}, 38(6):1353--1383, 2018.

\bibitem{bombod}
M.~Bomhoff and B.~Manthey.
\newblock Bisimplicial edges in bipartite graphs.
\newblock {\em Discrete Applied Mathematics}, 161:1699--1706, 2013.

\bibitem{BLS-gclass}
A.~Brandst{\"a}dt, V.~B. Le, and J.~P. Spinrad.
\newblock {\em {G}raph {C}lasses, {A} {S}urvey}.
\newblock SIAM Monographs on Discrete Mathematics and Applications,
  Philadelphia, 1999.

\bibitem{csorba}
P.~Csorba.
\newblock Subdivision yields {A}lexander duality on independence complexes.
\newblock {\em Electronic Journal of Combinatorics}, 16(2), 2009.

\bibitem{DHS-boundpd}
H.~Dao, C.~Huneke, and J.~Schweig.
\newblock Bounds on the regularity and projective dimension of ideals
  associated to graphs.
\newblock {\em Journal of Algebraic Combinatorics}, 38(1):37--55, 2013.

\bibitem{ODug}
O.~Duginov.
\newblock Partitioning the vertex set of a bipartite graph into complete
  bipartite subgraphs.
\newblock {\em Discrete Mathematics and Theoretical Computer Science},
  16(3):203--214, 2014.

\bibitem{AE}
A.~Engstr{\"{o}}m.
\newblock Complexes of directed trees and independence complexes.
\newblock {\em Discrete Mathematics}, 309(10):3299--3309, 2009.

\bibitem{FMPS}
H.~Fleischner, E.~Mujuni, D.~Paulusma, and S.~Szeider.
\newblock Covering graphs with few complete bipartite subgraphs.
\newblock {\em Theoretical Computer Science}, 410:2045--2053, 2009.

\bibitem{HMP}
P.~L. Hammer, F.~Maffray, and M.~Preissmann.
\newblock A characterization of chordal bipartite graphs.
\newblock RUTCOR Research Report, Rutgers University, New Brunswick, NJ, RRR,
  pp. 16–89, 1989.

\bibitem{AH-top}
A.~Hatcher.
\newblock {\em Algebraic Topology}.
\newblock Cambridge University Press, New York, 2006.

\bibitem{jiang}
W.~Jiang, T.~Liu, C.~Wang, and K.~Xu.
\newblock Feedback vertex sets on restricted bipartite graphs.
\newblock {\em Theoretical Computer Science}, 507:41–51, 2013.

\bibitem{MK-im}
M.~Katzman.
\newblock Characteristic-independence of {B}etti numbers of graph ideals.
\newblock {\em Journal of Combinatorial Theory, Series A}, 113(3):435--454,
  2006.

\bibitem{kozlov}
D.~Kozlov.
\newblock {\em Combinatorial Algebraic Topology}.
\newblock Springer, Berlin, Heidelberg, 2008.

\bibitem{villa}
R.~Villarreal.
\newblock {\em Monomial Algebras}.
\newblock Chapman and Hall/CRC, 2018.

\bibitem{russ}
R.~Woodroofe.
\newblock Matchings, coverings, and {Castelnuovo}-{Mumford} regularity.
\newblock {\em Journal of Commutative Algebra}, 6(2):287--304, 2014.

\bibitem{yetim}
M.~A. Yetim.
\newblock Independence complexes of strongly orderable graphs.
\newblock {\em Communications Faculty of Sciences, University of Ankara, Series
  A1 Mathematics and Statistics}, 71(2):445–455, 2022.

\end{thebibliography}

%%%%%%%%%%%%%%%%%%%%%%%%%%%%%%%%%%%%%%%%%%%%%%%%%%%%%%%

\end{document}